\documentclass[a4paper]{amsart}
\usepackage[utf8]{inputenc}
\usepackage{graphicx,tikz,eucal}
\usetikzlibrary{matrix,arrows}
\usepackage{tikz-cd}
\usetikzlibrary{calc}
\usepackage{amssymb}
\usepackage{lmodern}
\usepackage[T1]{fontenc}

\theoremstyle{plain}
  \newtheorem{thm}{Theorem}
  \newtheorem{defn}{Definition}
  \newtheorem{prop}{Proposition}

\theoremstyle{definition}
  \newtheorem{example}{Example}
  \newtheorem{rem}{Remark}

\newcommand{\mf}{\mathfrak}
\newcommand{\on}{\operatorname}

\newcommand{\g}{\mathfrak{g}}
\newcommand{\dd}{\mathfrak{d}}
\newcommand{\h}{\mathfrak{h}}
\newcommand{\D}{\mathcal D}
\newcommand{\C}{\mathcal C}
\newcommand{\Vect}{\mathcal Vect}
\newcommand{\Hom}{\operatorname{Hom}}
\newcommand{\la}{\langle}
\newcommand{\ra}{\rangle}

\newcommand{\id}{\on{id}}

\title{Quantization of Lie bialgebras revisited}
\author{Pavol \v{S}evera}
\address{Section of Mathematics, Universit\'{e} de Gen\`{e}ve, Geneva, Switzerland}
\email{pavol.severa@gmail.com}
\thanks{Supported in part by  the grant MODFLAT of the European Research Council and the NCCR SwissMAP of the Swiss National Science Foundation.}
\begin{document}
\maketitle

\begin{abstract}
We describe a new method of quantization of Lie bialgebras, based on a construction of Hopf algebras out of a cocommutative coalgebra and a braided comonoidal functor.
\end{abstract}

\section{Introduction}
The problem of functorial/universal quantization of Lie bialgebras was solved by Etingof and Kazhdan in \cite{EK}. They  quantize the double of the Lie bialgebra using a monoidal structure on the forgetful functor from the corresponding Drinfeld category and then define their quantization of the Lie bialgebra as a certain Hopf subalgebra of the Hopf algebra quantizing the double. Alternative solutions were given by Enriquez \cite{enr}, combining the approach of Etingof--Kazhdan with cohomological methods, and by Tamarkin \cite{tam}, based on formality of the little disks operad.

We present another solution of the problem. It is based on the fact that a cocommutative coalgebra and a braided comonoidal functor give rise, under certain invertibility conditions, to a Hopf algebra (Theorem \ref{thm:main}). This method avoids the need for quantization of the double, and is significantly simpler than the previous solutions, though it still needs a Drinfeld associator. It can also be used to quantize infinitesimally-braided Lie bialgebras to braided Hopf algebras.

There are two simple ideas behind this construction. The first one is a construction of non-commutative algebras (or coalgebras):  if $A_1$ and $A_2$ are associative algebra in a braided monoidal category then their tensor product $A_1\otimes A_2$ is an associative algebra; however, if $A_1$ and $A_2$ are commutative then $A_1\otimes A_2$ is not commutative in general. The same applies to the tensor product of coassociative coalgebras.

The second idea is that one can recover a group from its nerve. If $G$ is a group then we can consider the simplicial sets $E_nG:=G^{n+1}$ and $B_n G=(E_nG)/G$ ($E_\bullet G\to B_\bullet G$ is the simplicial model of the universal $G$-bundle). The group is $G=B_1G$ and the group multiplication on $G$ can be reconstructed from the face maps $B_2G\to B_1G$.

 If we replace $G$ by a cocommutative coalgebra $M$ in a braided monoidal category $\D$ then we can still define a simplicial coalgebra $E_n M:=M^{n+1}$ in $\D$, namely the cobar construction of the coalgebra $M$. As noted above, this simplicial coalgebra is not cocommutative. We then replace the operation $X\mapsto X/G$ with an appropriate braided comonoidal functor $F:\D\to\C$ to another braided monoidal category $\C$, set $B_nM:=F(E_nM)$ and construct (as a generalization of a group) a Hopf algebra structure on $B_1M=F(M\otimes M)$.


\section{Coalgebras in braided monoidal categories}
In this section we recall some basic definitions and facts concerning monoidal categories needed in the statement and the proof of Theorem \ref{thm:main}.
Additional details can be found in standard texts, such as \cite{maclane}.

A functor $F:\C_1\to\C_2$  between monoidal categories is a \emph{comonoidal functor}  (also called colax monoidal functor) if it comes equipped with a natural transformation 
$$c_{X,Y}:F(X\otimes Y)\to F(X)\otimes F(Y)$$
 such that
$$
\begin{tikzcd}
F((X\otimes Y)\otimes Z)\arrow{r}\arrow{d}& F(X\otimes Y)\otimes F(Z)\arrow{r} &(F(X)\otimes F(Y))\otimes F(Z)\arrow{d}\\
F(X\otimes (Y\otimes Z))\arrow{r}& F(X)\otimes F( Y\otimes Z)\arrow{r} &F(X)\otimes (F(Y)\otimes F(Z))
\end{tikzcd}
$$
commutes, together with a morphism
$$c:F(1_{\C_1})\to1_{\C_2}$$
compatible with the units. Let us stress that $c_{X,Y}$ and $c$ are \emph{not} required to be isomorphisms.  If $c_{X,Y}$'s and $c$ \emph{are} isomorphisms then  $F$ is a \emph{strongly monoidal functor}.

  A comonoidal functor $F:\C_1\to\C_2$  sends coalgebras to coalgebras: if $M\in\C_1$ is a coalgebra then the coproduct on $F(M)$ is the composition 
$$F(M)\to F(M\otimes M)\to F(M)\otimes F(M)$$
and the counit is the composition
$$F(M)\to F(1_{\C_1})\to1_{\C_2}.$$

If $\C$ is a braided monoidal category then the functor 
$$\otimes:\C\times\C\to\C$$
is a strongly monoidal functor, with the monoidal structure 
$$(X_1\otimes Y_1)\otimes(X_2\otimes Y_2)\to (X_1\otimes X_2)\otimes(Y_1\otimes Y_2)\quad (\forall X_1,X_2,Y_1,Y_2\in\C)$$
given by the parenthesized braid
\begin{equation}\label{eq:parb}
\begin{tikzpicture}[baseline=1cm]
\coordinate (diff) at (0.65,0);
\coordinate (dy) at (0,0.5);
\node(X1) at (0,0) {$(X_1$};
\node(Y1) at ($(X1)+(diff)$) {$Y_1)$};
\node(Z1) at (2,0) {$(X_2$};
\node(W1) at ($(Z1)+(diff)$) {$Y_2)$};
\node(X2) at (0,2) {$(X_1$};
\node(Z2) at ($(X2)+(diff)$) {$X_2)$};
\node(Y2) at (2,2) {$(Y_1$};
\node(W2) at ($(Y2)+(diff)$) {$Y_2)$};
\draw(X1)--(X2);
\draw(W1)--(W2);
\draw[line width=1ex,white] (Z1)..controls +(0,1) and +(0,-1)..(Z2);
\draw (Z1)..controls +(0,1) and +(0,-1)..(Z2);
\draw[line width=1ex,white] (Y1)..controls +(0,1) and +(0,-1)..(Y2);
\draw(Y1)..controls +(0,1) and +(0,-1)..(Y2);
\end{tikzpicture}
\end{equation}
In particular, if $M$ an $M'$ are coalgebras  in $\C$ then $M\otimes M'$ is a coalgebra as well, with the coproduct 
\begin{equation}\label{eq:fus}
\begin{tikzpicture}[baseline=-1cm]
\coordinate (diff) at (0.7,0);
\coordinate (dy) at (0,-0.8);
\node(A1) at (0,0) {$(M$};
\node(B1) at ($(A1)+(diff)$) {$M')$};
\node(A2) at (2,0) {$(M$};
\node(B2) at ($(A2)+(diff)$) {$M')$};
\node(A3) at ($(A1)+(0,-2.5)+0.5*(diff)$) {$M$};
\node(B3) at ($(A2)+(0,-2.5)+0.5*(diff)$) {$M'$};
\coordinate(A) at ($(A3)-(dy)$);
\coordinate(B) at ($(B3)-(dy)$);
\draw[line width=1ex,white] (B1)..controls +(0,-1) and +(0,1)..(B);
\draw (B1)..controls +(0,-1) and +(0,1)..(B);
\draw[line width=1ex,white]  (A2)..controls +(0,-1) and +(0,1)..(A);
\draw (A2)..controls +(0,-1) and +(0,1)..(A);
\draw (A1)..controls +(0,-1) and +(0,1)..(A);
\draw (B2)..controls +(0,-1) and +(0,1)..(B);
\draw (B)--(B3);
\draw(A)--(A3);
\end{tikzpicture}
\end{equation}
In this way the category of coalgebras in $\C$ becomes a monoidal category. If $F:\C_1\to\C_2$ is a braided comonoidal functor, i.e.\ if $F$ is a monoidal functor such that the diagram (where $\beta$ is the braiding)
$$
\begin{tikzcd}
F(X\otimes Y)\arrow{r}\arrow{d}[swap]{F(\beta^{\C_1}_{X,Y})}& F(X)\otimes F(Y)\arrow{d}{\beta^{\C_2}_{F(X),F(Y)}}\\
F(Y\otimes X)\arrow{r}&F(Y)\otimes F(X)
\end{tikzcd}
$$
commutes,
 then $F(M\otimes M')\to F(M)\otimes F(M')$ is a morphism of coalgebras; $F$ thus becomes a comonoidal functor from the category of coalgebras in $\C_1$ to the category of coalgebras in $\C_2$.

An algebra (i.e.\ a monoid) $H$ in the category of coalgebras in $\C$ is called a \emph{bialgebra} in $\C$. It is a \emph{Hopf algebra} if it comes with an \emph{invertible} morphism $S\in\Hom_\C(H,H)$ such that 
$$m_H\circ(S\otimes\id_H)\circ\Delta_H=m_H\circ(\id_H\otimes S)\circ\Delta_H=\eta_H\circ\epsilon_H$$
where $\epsilon_H:H\to1_\C$ is the counit, $\eta_H:1_\C\to H$ the unit, and $m_H$ and $\Delta_H$ the product and the coproduct of $H$.

\section{A construction of Hopf algebras}

In this section we shall construct a Hopf algebra out of a braided comonoidal functor and of a cocommutative coalgebra. The rest of the paper is an application of this construction.

Suppose $M$ is a coalgebra in a braided monoidal category $\D$.
Even if $M$ is cocommutative (i.e.\ if $\beta_{M,M}\circ\Delta_M=\Delta_M$, where $\Delta_M:M\to M\otimes M$ is the coproduct and $\beta_{M,M}:M\otimes M\to M\otimes M$ the braiding), the coalgebra $M\otimes M$ may be non-cocommutative. This simple observation will be our source of non-cocommutativity.
If $F:\D\to\C$ is a braided comonoidal functor to another braided monoidal category $\C$ then the coalgebra $F(M)$ is cocommutative, but $F(M\otimes M)$ might be not.

As we shall see, under certain compatibility conditions on $M$ and $F$, the coalgebra $F(M\otimes M)$ is a Hopf algebra.

\begin{defn}
Let $\D$ and $\C$ be  braided monoidal categories and $(M,\Delta_M,\epsilon_M)$ a cocommutative coalgebra in $\D$. A braided comonoidal functor
$$F:\D\to\C$$
 is \emph{$M$-adapted} if it satisfies these invertibility conditions: the composition
$$F(M)\xrightarrow{F(\epsilon_M)}F(1_\D)\to1_\C$$
is an isomorphism, and  for every objects $X,Y\in\D$ the morphism
$$\tau^{(M)}_{X,Y}:F((X\otimes M)\otimes Y)\to F(X\otimes M)\otimes F(M\otimes Y),$$
defined as the composition
\begin{multline*}
F((X\otimes M)\otimes Y)\xrightarrow{F((\id_X\otimes\Delta_M)\otimes\id_Y)}
F\bigl((X\otimes (M\otimes M))\otimes Y\bigr)\cong \\
\cong F\bigl((X\otimes M)\otimes(M\otimes Y)\bigr)\to F(X\otimes M)\otimes F(M\otimes Y),
\end{multline*}
is an isomorphism.
\end{defn}

\begin{rem}
The functor $\D\to\D$, $X\mapsto M\otimes X$, is comonoidal (since $\otimes:\D\times\D\to\D$ is strongly monoidal and $M$ is a coalgebra); explicitly, the comonoidal structure $M\otimes (X\otimes Y)\to (M\otimes X)\otimes(M\otimes Y)$ is
\begin{equation}\label{eq:tw}
\begin{tikzpicture}[baseline=1cm]
\coordinate (diff) at (0.65,0);
\coordinate (dy) at (0,0.5);
\node(Y1) at (0.7,0) {$M$};
\node(Z1) at (2,0) {$(X$};
\node(W1) at ($(Z1)+(diff)$) {$Y)$};
\node(X2) at (1,2) {$(M$};
\node(Z2) at ($(X2)+(diff)$) {$X)$};
\node(Y2) at (2.5,2) {$(M$};
\node(W2) at ($(Y2)+(diff)$) {$Y)$};
\coordinate (Delta) at ($(Y1)+(dy)$);
\draw(W1)..controls +(0,1) and +(0,-1)..(W2);

\draw[line width=1ex,white] (Z1)..controls +(0,1) and +(0,-1)..(Z2);
\draw (Z1)..controls +(0,1) and +(0,-1)..(Z2);
\draw[line width=1ex,white] (Delta)..controls +(0,0.7) and +(0,-0.7)..(Y2);
\draw(Delta)..controls +(0,0.7) and +(0,-0.7)..(Y2);

\draw (Delta)..controls +(0,0.7) and +(0,-0.7)..(X2);
\draw (Y1)--(Delta);
\end{tikzpicture}
\end{equation}
A braided comonoidal functor $F$ is $M$-adapted iff the composition $\D\xrightarrow{M\otimes}\D\xrightarrow{F}\C$, which is a priori comonoidal, is in fact strongly monoidal.
\end{rem}

\begin{thm}\label{thm:main}
Let $F:\D\to\C$ be an $M$-adapted functor. Then $F(M\otimes M)$ is a Hopf algebra in $\C$, with the structure given as follows.
\begin{itemize}
\item The coalgebra structure on $F(M\otimes M)$ is inherited from the coalgebra structure  on $M\otimes M$, with the coproduct \eqref{eq:fus} and counit $\epsilon_M\otimes\epsilon_M$.
\item The product on $F(M\otimes M)$ is the composition
\begin{multline}\label{eq:mult-def}
F(M\otimes M)\otimes F(M\otimes M)\xrightarrow{{\tau^{(M)}_{M,M}}^{-1}}F((M\otimes M)\otimes M)\\ \xrightarrow{F(\id_M\otimes\epsilon_M\otimes\id_M)}F(M\otimes M).
\end{multline}
\item The unit is
\begin{equation}\label{eq:unit-def}
1_\C\cong F(M)\xrightarrow{F(\Delta_M)}F(M\otimes M).
\end{equation}
\item The antipode is 
$$F(M\otimes M)\xrightarrow{F(\beta_{M,M})}F(M\otimes M)$$
where $\beta_{M,M}:M\otimes M\to M\otimes M$ is the braiding in $\D$.
\end{itemize}
\end{thm}
\begin{proof}
To simplify notation, let us replace $\D$ and $\C$ with  equivalent strict monoidal categories. The sequence of objects $X_n:=M^{\otimes(n+1)}$ ($n=0,1,2,\dots$) is a simplicial object of $\D$, with degeneracies ${\id_M^{\otimes k}}\otimes\Delta_M^{\vphantom\otimes}\otimes\id_M^{\otimes(n-k)}$ and faces ${\id_M^{\otimes k}}\otimes\epsilon_M^{\vphantom\otimes}\otimes\id_M^{\otimes(n-k)}$ (it is the cobar construction of the coalgebra $M$). Since $M$ is cocommutative, $\Delta_M:M\to M\otimes M$ is a morphisms of coalgebras, and thus $X_\bullet$ is a simplicial coalgebra in $\D$. As a result (using comonoidality of $F$), $Y_\bullet:=F(X_\bullet)$ is a simplicial coalgebra in $\C$.

By repeatedly using invertibility of $\tau^{(M)}_{X,Y}$'s we know that the composition
$$Y_n=F(M^{\otimes(n+1)})\xrightarrow{F(\id_M^{\vphantom\otimes}
\otimes\Delta_M^{\otimes(n-1)}\otimes\id_M^{\vphantom\otimes})} F(M^{\otimes2n})\to F(M\otimes M)^{\otimes n}$$
is an isomorphism. Since $M$ is cocommutative and $F$ is braided, both arrows are morphisms of coalgebras, hence we have an isomorphism of coalgebras in $\C$
$$Y_n\cong F(M\otimes M)^{\otimes n}.$$
In terms of this isomorphism, the face maps of $Y_\bullet$ are given by the product \eqref{eq:mult-def} on $F(M\otimes M)$ and the degeneracy maps are given by including the unit \eqref{eq:unit-def} of $F(M\otimes M)$. This shows that the product is  associative  and that the unit is a unit of the product (and that $Y_\bullet$ is the bar construction of the resulting algebra $F(M\otimes M)$).
In more detail, associativity follows from the fact that the map
$$F(M\otimes M)^{\otimes 3}\cong Y_3=F(M^{\otimes 4})\xrightarrow{F(\id_M\otimes\epsilon_M\otimes\epsilon_M\otimes\id_M)}Y_1=F(M\otimes M)$$
is equal to both ways of bracketing of the product, and the fact that the unit \eqref{eq:unit-def} is indeed a unit of the product is simply the fact that the compositions
$$F(M\otimes M)\xrightarrow{F(\Delta_M\otimes\id_M)} F(M\otimes M\otimes M)\xrightarrow{F(\id_M\otimes\epsilon_M\otimes\id_M)} F(M\otimes M)$$
$$F(M\otimes M)\xrightarrow{F(\id_M\otimes\Delta_M)} F(M\otimes M\otimes M)\xrightarrow{F(\id_M\otimes\epsilon_M\otimes\id_M)} F(M\otimes M)$$
are the identity.
The object $F(M\otimes M)$ is thus an algebra in the category of coalgebras in $\C$, i.e.\ it is a bialgebra in $\C$.

Finally, the fact that $F(\beta_{M,M})$ is an antipode for the bialgebra $F(M\otimes M)$ follows easily from the definitions.
\end{proof}

\begin{rem}
One can summarize the proof as follows. There is a simple way to recognize nerves of bialgebras among all simplicial coalgebras: if $Y_\bullet$ is a simplicial coalgebra in $\C$ then it is the nerve of a bialgebra in $\C$ iff for any $k$ and $l$ the morphism in $\C$
$$Y_{k+l}\to Y_k\otimes Y_l$$
obtained as the composition $Y_{k+l}\xrightarrow{\Delta} Y_{k+l}\otimes Y_{k+l}\xrightarrow{f\otimes g} Y_k\otimes Y_l$, where $f$ and $g$ are the first $k$-face and last $l$-face maps respectively, is an isomorphism of \emph{coalgebras}, and if $\epsilon:Y_0\to 1_\C$ is also an isomorphism of coalgebras. The bialgebra is then $Y_1$. We simply verified these conditions for our $Y_\bullet$, and finally checked that $Y_1$ is actually a Hopf algebra.

\end{rem}

\begin{rem}
The proof used simplicial methods in an essential way. It is not difficult to extract from it a direct proof:  the associativity of the product 
$$m:F(M\otimes M)\otimes F(M\otimes M)\to F(M\otimes M)$$
 follows from the commutative diagram
$$
\begin{tikzcd}
F(MM)F(MM)F(MM) \arrow{r}[swap]{|\tau^{-1}}\arrow{d}{\tau^{-1}|}\arrow[bend left=20]{rr}[swap]{|m}\arrow[bend right=60,start anchor={[xshift = -2em]}, end anchor={[xshift = -2em]}]{dd}[swap]{m|} & F(MM)F(MMM)\arrow{r}[swap]{|F(|\epsilon|)}\arrow{d}{\tau^{-1}} & F(MM)F(MM)\arrow{d}{\tau}\arrow[bend left,start anchor={[xshift = 2em]}, end anchor={[xshift = 2em]}]{dd}{m}\\
F(MMM)F(MM)\arrow{r}{\tau^{-1}}\arrow{d}{F(|\epsilon|)|} & F(MMMM)\arrow{r}{F(||\epsilon|)}\arrow{d}{F(|\epsilon||)}\arrow{rd}{F(|\epsilon\epsilon|)} & F(MMM)\arrow{d}{F(|\epsilon|)}\\
F(MM)F(MM)\arrow{r}{\tau^{-1}}\arrow[bend right=20]{rr}{m} & F(MMM)\arrow{r}{F(|\epsilon|)} & F(MM)
\end{tikzcd}
$$
where,  to keep its size reasonable, $\otimes$'s and indexes are dropped and $\id$'s are denoted by $|$.
\end{rem}

\begin{rem}
Let $\D_{univ}$ be the universal braided monoidal category with a chosen cocommutative coalgebra: the objects of $\D_{univ}$ are tensor powers of the coalgebra, the morphisms can be visualized as parenthesized braids with strands attached non-bijectively at the bottom. One can show that for any Hopf algebra $H$ in any braided monoidal category $\C$ there is an $M$-adapted functor $F:\D_{univ}\to\C$ (where $M\in\D_{univ}$ is the coalgebra) such that $H=F(M\otimes M)$ as a Hopf algebra, and also that $F$ is unique up to an isomorphism. 

This gives us a description/construction of the braided (or ordinary) PROP of braided (or ordinary) Hopf algebras as the universal braided (or symmetric) category $\C_{univ}$ admitting a $M$-adapted functor $F_{univ}:\D_{univ}\to\C_{univ}$ (i.e.\ such that any $M$-adapted functor $F:\D_{univ}\to \C$ factors via $F_{univ}$ through a strict braided (or symmetric) monoidal functor $\C_{univ}\to\C$).
\end{rem}

\begin{rem}
As we noticed above, the functor $\D\to\C$, $X\mapsto F(M\otimes X)$, is strongly monoidal, with the monoidal structure given by \eqref{eq:tw}. The functor can actually  be seen as as a braided strongly monoidal functor from $\D$ to the center of the monoidal category of $F(M\otimes M)$-modules in $\C$. Namely,  $F(M\otimes X)$  is a $F(M\otimes M)$-module via
$$
F(M\otimes M)\otimes F(M\otimes X)\xrightarrow{{{\tau^{(M)}_{M,X}}^{-1}}}
 F((M\otimes M)\otimes X)\xrightarrow{F(\id_M\otimes\epsilon_M\otimes\id_X)}F(M\otimes X)
$$
and the half-braiding $$F(M\otimes X)\otimes R\to R\otimes F(M\otimes X)$$
 for  $F(M\otimes M)$-modules $R$  is given by the morphism
$$q_X:F(M\otimes X)\to F(M\otimes X)\otimes F(M\otimes M)$$
represented by the diagram (with $F$ and its comonoidal structure depicted in gray)
$$
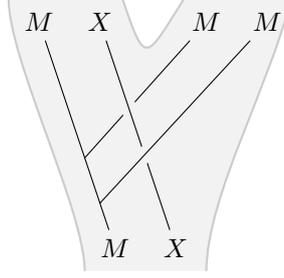
\begin{tikzpicture}
\fill[black!5] (-0.4,-0.3)--(1.2,-0.3)..controls +(0,1) and +(-0.2,-1)..(2.3,3.3)--(0.9,3.3)..controls+(-0.5,-0.8)and+(0.3,-0.9)..(0.1,3.3)--(-1.4,3.3)..controls+(0,-0.8)and+(0,0.8)..cycle;
\draw[black!20,thick](1.2,-0.3)..controls +(0,1) and +(-0.2,-1)..(2.3,3.3) (0.9,3.3)..controls+(-0.5,-0.8)and+(0.3,-0.9)..(0.1,3.3) (-1.4,3.3)..controls+(0,-0.8)and+(0,0.8)..(-0.4,-0.3);
\node(M) at (0,0) {$M$};
\node(X) at (0.8,0) {$X$};
\node(M1) at (-1,3) {$M$};
\node(M2) at (1.2,3) {$M$};
\node(M3) at (2,3) {$M$};
\node(X1) at (-0.2,3) {$X$};
\coordinate(S1) at (-0.2,0.6);
\coordinate(S2) at (-0.4,1.2);
\draw (S2)--(M2);
\draw [line width=2mm,black!5] (X)--(X1);
\draw (X)--(X1);
\draw [line width=2mm,black!5] (S1)--(M3);
\draw  (S1)--(M3) (M)--(M1) ;
\end{tikzpicture}
$$
as the composition
\begin{multline*}
F(M\otimes X)\otimes R\xrightarrow{q_X\otimes\id_R} F(M\otimes X)\otimes F(M\otimes M)\otimes R\\
\to F(M\otimes X)\otimes R\xrightarrow{\beta^\C_{F(M\otimes X),R}}R\otimes F(M\otimes X)
\end{multline*}

\end{rem}

\section{Infinitesimally braided categories}
In this section we recall Drinfeld's construction of braided monoidal categories via associators. We also observe how cocommutative coalgebras and braided comonoidal functors arise in this construction, as we need to feed them to Theorem \ref{thm:main}.

Let us fix  a field $K$  with $\on{char}K=0$. By a $K$-linear category we mean a category enriched over $K$-vector spaces, i.e.\ $\Hom(X,Y)$ should be a vector space over $K$ and the composition map $\Hom(X,Y)\times\Hom(Y,Z)\to\Hom(X,Z)$ should be bilinear.

An \emph{infinitesimally braided category} (\emph{i-braided category} for short) is a $K$-linear symmetric monoidal category $\C$ together with a natural transformation
$$t_{X,Y}:X\otimes Y\to X\otimes  Y$$
such that 
$$t_{X,Y\otimes Z}=t_{X,Y}\otimes\id_Z+(\id_X\otimes\sigma_{Z,Y})\circ(t_{X,Z}\otimes\id_Y)\circ(\id_X\otimes\sigma_{Y,Z})$$
(where $\sigma_{Y,Z}:Y\otimes Z\to Z\otimes Y$ is the symmetry) and
\begin{align*}
t_{Y,X}\circ\sigma_{X,Y}&=\sigma_{X,Y}\circ t_{X,Y}\\
  t_{X,1_\C}&=0.
\end{align*}
The transformation $t_{X,Y}$ defines a deformation of the symmetric monoidal structure of $\C$ to a braided monoidal structure: if $\epsilon$ is a formal parameter with $\epsilon^2=0$ and $\C_\epsilon$ is the same as $\C$ but with $\Hom_{\C_\epsilon}(X,Y)=\Hom_\C(X,Y)[\epsilon]$ then
$$\beta_{X,Y}=\sigma_{X,Y}\circ(\id_{X\otimes Y}+\epsilon\,t_{X,Y}/2)$$
is a braiding on $\C_\epsilon$ (with the monoidal structure inherited from $\C$).

\begin{example}
Let $\dd$ be a Lie algebra over $K$ and let $t\in(S^2\dd)^\dd$. The category of $U\dd$-modules is infinitesimally braided, with $t_{X,Y}$ given by the action of $t\in\dd\otimes\dd\subset U\dd\otimes U\dd$.
\end{example}

Let  $\C$ be an i-braided category and let $\C_\hbar$ be as $\C$, with $\Hom_{\C_\hbar}(X,Y)=\Hom_\C(X,Y)[\![\hbar]\!]$. Following Drinfeld \cite{drinfeld}, we can make $\C_\hbar$ to a braided monoidal category (extending the first order deformation $\C_\epsilon$) in the following way. Let
$$\Phi\in K\la\!\la x,y\ra\!\ra$$
be an element which is group-like 
 w.r.t\ the coproduct
$$\Delta x=x\otimes 1+1\otimes x,\ \Delta y=y\otimes 1+1\otimes y.$$
Let us define a new braiding on $\C_\hbar$ by
$$\beta_{X,Y}=\sigma_{X,Y}\circ e^{\hbar t_{X,Y}/2}$$
and a new associativity constraint $\gamma_{X,Y,Z}$ by
$$(X\otimes Y)\otimes Z\xrightarrow{\Phi(\hbar t_{X,Y},\hbar t_{Y,Z})}(X\otimes Y)\otimes Z\to X\otimes(Y\otimes Z).$$
\begin{rem}
If $\C$ is enriched over coalgebras  and $t_{X,Y}$ are primitive then the new braidings and associativity constraints are group-like. This is the reason for demanding $\Phi$ to be group-like and also for choosing $e^{\hbar t_{X,Y}/2}$ among all power series $1+\hbar t_{X,Y}/2+\dots$ in $t_{X,Y}$. We shall not need this fact in what follows.
\end{rem}

The pentagon and hexagon relations for $\beta_{X,Y}$'s and $\gamma_{X,Y,Z}$ translate to the following properties of $\Phi$:

\begin{prop}[\cite{drinfeld}]
 The category $\C_\hbar$ with the natural transformations $\beta$ and $\gamma$ is a braided monoidal category provided $\Phi$ is a \emph{Drinfeld associator}, i.e.\ if it satisfies the relations
$$\Phi(y,x)=\Phi(x,y)^{-1},$$
$$e^{x/2}\,\Phi(y,x)\,e^{y/2}\,\Phi(z,y)\,e^{z/2}\,\Phi(x,z)=1\;\text{ where }z=-x-y,$$
$$\Phi^{2,3,4}\,\Phi^{1,23,4}\,\Phi^{1,2,3}=\Phi^{1,2,34}\,\Phi^{12,3,4}.$$
\end{prop}
\noindent The last relation takes place in the algebra generated by symbols 
$$ t^{i,j}\quad (1\leq i,j\leq 4, i\neq j,\ t^{i,j}=t^{j,i})$$ modulo the relations 
$$[t^{i,j},t^{i,k}+t^{j,k}]=0\text{ and }[t^{i,j},t^{k,l}]=0\text{ if }\{i,j\}\cap\{k,l\}=\emptyset,$$
and $\Phi^{A,B,C}:=\Phi(t^{A,B},t^{B,C})$ with $t^{A,B}:=\sum_{i\in A,j\in B}t^{i,j}$. See \cite{drinfeld} for details, and also for a proof of existence of Drinfeld associators for every $K$.

We shall denote the category $\D_\hbar$ with its new braided monoidal structure by $\D^\Phi_\hbar$.

Let us now define infinitesimal versions of commutative coalgebras and of braided comonoidal functors.

\begin{defn}
Let $\D$ be an i-braided category. An \emph{i-cocommutative coalgebra} in $\D$ is an object $M$ which is a cocommutative coalgebra in the symmetric monoidal category $\D$, and which satisfies $t_{M,M}\circ\Delta_M=0$. If $\C$ is another i-braided category, an \emph{i-braided comonoidal functor} $F:\D\to\C$ is a $K$-linear symmetric comonoidal functor $F:\D\to\C$ such that 
$$
\begin{tikzcd}
F(X\otimes Y)\arrow{r}\arrow{d}[swap]{F(t^\D_{X,Y})}& F(X)\otimes F(Y)\arrow{d}{t^\C_{F(X),F(Y)}}\\
F(X\otimes Y)\arrow{r}&F(X)\otimes F(Y)
\end{tikzcd}
$$
commutes. $F$ is \emph{$M$-adapted} if it is $M$-adapted as a braided comonoidal functor between the symmetric monoidal categories $\D$ and $\C$.
\end{defn}

\begin{prop}\label{prop:drinf}
Let $\D$ and $\C$ be i-braided categories. Let $\Phi$ be a Drinfeld associator. If $M\in\D$ is an i-cocommutative coalgebra then it is, with the same coproduct and counit, a cocommutative coalgebra in $\D^\Phi_\hbar$. If $F:\D\to\C$ is an i-braided comonoidal functor then it is, with the same comonoidal structure, a braided comonoidal functor $\D^\Phi_\hbar\to\C^\Phi_\hbar$. If $F:\D\to\C$ is $M$-adapted then it remains $M$-adapted as a functor $\D^\Phi_\hbar\to\C^\Phi_\hbar$.
\end{prop}
\begin{proof}
If $F:\D\to\C$ is i-braided comonoidal then the braided comonoidality of $F:\D^\Phi_\hbar\to\C^\Phi_\hbar$ is an immediate consequence of the definitions.

Let $\mathbf 1$ be the symmetric monoidal category with a unique object $I$ and with $\Hom(I,I)=K$. Let us make it i-braided via $t_{I,I}=0$. An i-cocommutative coalgebra $M\in\D$ is equivalent to an i-braided comonoidal functor $G:\mathbf 1\to\D$, with $M=G(I)$. The functor $G$ is thus braided comonoidal as a functor $\mathbf 1_\hbar=\mathbf 1^\Phi_\hbar\to\D^\Phi_\hbar$, which means that $G(I)=M$ is a cocommutative coalgebra in $\D^\Phi_\hbar$.
\end{proof}

\begin{example}\label{ex:lba}
Let $\dd$ be a Lie algebra and $t\in(S^2\dd)^\dd$. Let us suppose that $\dim\mf d<\infty$ and that  $t$ is non-degenerate, and thus defines a symmetric pairing $\la,\ra:\dd\times\dd\to K$. Let $\g\subset\dd$ be a Lie subalgebra which is Lagrangian w.r.t.\ the pairing, i.e.\ $\g^\perp=\g$. Then the functor
$$F:U\dd\text{-mod}\to\mathcal{V}ect,\ F(V)= V/(\g\cdot V),$$
with 
$$c_{V,W}:(V\otimes W)/(\g\cdot(V\otimes W))\to\bigl(V/(\g\cdot V)\bigr)\otimes \bigl(W/(\g\cdot W)\bigr)$$
being the natural projection,
is i-braided comonoidal, as the projection of $t$ to $S^2(\dd/\g)$ vanishes.

If $\g^*\subset\dd$ is another Lagrangian Lie subalgebra, such that $\g\cap\g^*=0$, i.e.\ if $(\g,\g^*\subset\dd)$ is a Manin triple, then 
$$M=U\dd/(U\dd)\g^*,$$
with the coalgebra structure inherited from $U\dd$, is i-cocommutative. The reason is again that the image of $t$ in  $S^2(\dd/\g^*)$ vanishes. The functor $F$ is $M$-adapted.

The Hopf algebra $F(M\otimes M)$ in $\mathcal Vect_\hbar$ (given by Proposition \ref{prop:drinf} and Theorem \ref{thm:main}) is a quantization of the Lie bialgebra $\g$. We  discuss this quantization in detail in Section \ref{sect:qlb}.
\end{example}

\begin{example}\label{ex:blba}
More generally, let $\mf p\subset\dd$ be a coisotropic Lie subalgebra, i.e.\ such that $\mf p^\perp\subset\mf p$. Notice that $\mf p^\perp$ is an ideal in $\mf p$. Let $\h=\mf p/\mf p^\perp$. The functor
$$F:U\dd\text{-mod}\to U\h\text{-mod},\quad F(V)=V/(\mf p^\perp\cdot V)$$
is i-braided comonoidal.

Let $\bar{\mf p}\subset\dd$ be another coisotropic Lie subalgebra such that $\dd=\bar{\mf p}\oplus \mf p^\perp$ as a vector space (for example, $\dd$ can be semisimple with the Cartan-Killing form and $\mf p$ and $\bar{\mf p}$ a pair of opposite parabolic subalgebras; $\mf p^\perp\subset\mf p$ is then the nilpotent radical of $\mf p$). If we set
$$M=U\dd/(U\dd)\bar{\mf p}$$
then $F$ is $M$-adapted.

Proposition \ref{prop:drinf} and Theorem \ref{thm:main} now make $F(M\otimes M)$ to a Hopf algebra in the braided monoidal category $U\h\text{-mod}_\hbar^\Phi$. The object $F(M\otimes M)$ can be naturally identified with $U(\mf p^\perp)$, and gives us a deformation of the standard Hopf algebra structure on $U(\mf p^\perp)$.
\end{example}

\section{Quantization of Lie bialgebras}\label{sect:qlb}

Let us recall that if $\g$ is a Lie algebra and if $m_\hbar$ and $\Delta_\hbar$ are formal deformation of the product and of the coproduct on $U\g$, making $U\g$ (with the original unit and counit) to a bialgebra in $\mathcal Vect_\hbar$, then
$$\delta(x):=(\Delta_\hbar-\Delta^{op}_\hbar)(x)/\hbar\ \text{ mod }\hbar,\quad x\in\g\subset U\g,$$
$$\delta:\g\to\g\otimes\g,$$
is a Lie cobracket and that it makes $\g$ to a Lie bialgebra. The quantization problem of Lie bialgebras is the problem of constructing $m_\hbar$ and $\Delta_\hbar$ out of $[,]$ and $\delta$ in a functorial/universal way.

To solve the problem, let us
reformulate Example \ref{ex:lba} so that it works for infinite-dimensional Lie bialgebras.
Let $\g$ be a Lie bialgebra over a field $K$ of characteristic $0$, with cobracket $\delta:\g\to\g\otimes\g$.
Let $\D$ be the category of $\g$-dimodules, i.e.\ of vector spaces with an action of the Lie algebra $\g$
$$\rho:\g\otimes V\to V$$
and with a  coaction of the Lie coalgebra  $\g$
$$\tilde\rho:V\to \g\otimes V$$
such that
$$(\id\otimes\rho)\circ\sigma_{1,2}\circ(\id\otimes\tilde\rho)=\tilde\rho\circ\rho-
(\id\otimes\rho)\circ(\delta\otimes\id)+
(\id\otimes[,])\circ(\id\otimes\tilde\rho),$$
or equivalently, such that the resulting map
$$(\rho+\tilde\rho^*):(\g\oplus\g^*)\otimes V\to V$$
is an action of the double $\dd=\g\oplus\g^*$. If $\dim\g<\infty$ then $\D$ is simply the category of $U\dd$-modules; in general it is its (full) subcategory.

 The category $\D$ is i-braided, with
$$t_{V,W}:V\otimes W\to V\otimes W$$
given in terms of $\rho$ and $\tilde\rho$ as
$$t_{V,W}=(\id_V\otimes\rho_W)\circ\sigma_{1,2}\circ(\tilde\rho_V\otimes\id_W)+(\rho_V\otimes\id_W)\circ\sigma_{1,2}\circ(\id_V\otimes\tilde\rho_W).$$

Let us now define an i-cocommutative coalgebra $M$ in $\D$. Let
$$M=U\g$$
with the $\g$-action
$$\rho_M(x\otimes y)=xy\quad(x\in\g,y\in U\g)$$
and with the coaction $\tilde\rho_M$ uniquely determined by
\begin{equation}\label{trhoM}
\tilde\rho_M(1)=0;
\end{equation}
in particular,
$$\tilde\rho_M(x)=\delta(x)\text{ for }x\in\g.$$
The usual coproduct $\Delta:M\to M\otimes M$ and counit $\epsilon:M\to K$ of $U\g$ make $M$ to an i-cocommutative coalgebra in $\D$. To verify the i-cocommutativity of $M$ it's enough to check that $t_{M,M}(\Delta 1)=0$, which follows from \eqref{trhoM} as $\Delta(1)=1\otimes 1\in M\otimes M$.

\begin{rem}
The coalgebra $M$ plays an important role in the quantization of Etingof and Kazhdan \cite{EK}, where it is denoted $M_-$. Despite this similarity, the relation between these two quantization methods is unclear to me. For technical reasons Etingof and Kazhdan had to replace $\D$ with a somewhat complicated category of equicontinuous $U\mf d$-modules.
\end{rem}

Let $\mathcal Vect$ denote the symmetric monoidal category of vector spaces over $K$ (we can see it as i-braided with $t_{X,Y}=0$ for all $X,Y\in\mathcal Vect$). Let $F:\D\to\mathcal Vect$ be given by
$$F(V)=V/(\g\cdot V).$$
It is an i-braided comonoidal functor, which is $M$-adapted, with the  comonoidal structure given by the projection
$$(V\otimes W)/\bigl(\g\cdot(V\otimes W)\bigr)\to V/(\g\cdot V)\otimes W/(\g\cdot W).$$

We have a linear bijection
\begin{subequations}\label{eq:ident}
\begin{equation}
F(M\otimes M)\cong U\g
\end{equation}
given by
\begin{equation}\label{eq:ident2}
[x\otimes y]\mapsto S_0(x)y \quad(x,y\in U\g)
\end{equation}
\end{subequations}
where $S_0$ is the usual antipode on $U\g$ and $[x\otimes y]$ denotes the class of $x\otimes y\in M\otimes M$ in $F(M\otimes M)$. The inverse of this bijection is given by
$x\mapsto [1\otimes x]$, $x\in U\g$.

Let is now choose a Drinfeld associator $\Phi$ over $K$ and consider the braided monoidal category $\D^\Phi_\hbar$. By Theorem \ref{thm:main} and Proposition \ref{prop:drinf}, $F(M\otimes M)$ becomes a Hopf algebra in $\mathcal Vect_\hbar^\Phi=\mathcal Vect_\hbar$.

\begin{thm}\label{thm:qlb}
The Hopf algebra structure on $F(M\otimes M)\cong U\g$ in $\mathcal Vect_\hbar$ is a deformation of the cocommutative Hopf algebra $U\g$, and its classical limit is the Lie bialgebra structure on $\g$. 
\end{thm}
\begin{proof}
Let us identify $U\g\otimes U\g$ and $F((M\otimes M)\otimes M)$  via the linear map
$$x\otimes y\mapsto [S_0(x)\otimes 1\otimes y]\quad (x\otimes y\in U\g\otimes U\g).$$
The isomorphism in $\mathcal Vect_\hbar$
$$\tau^{(M)}_{M,M}:F((M\otimes M)\otimes M)\to F(M\otimes M)\otimes F(M\otimes M)$$
then  becomes an isomorphism (of vector spaces)
$$U\g\otimes U\g\to U\g\otimes U\g,$$
which is, by the definition of $\tau^{(M)}_{M,M}$ and by \eqref{eq:ident}, of the form
$$x\otimes y\mapsto x\otimes y+O(\hbar^2)$$
(as $\Phi(\hbar t^{1,2}\!,\hbar t^{2,3})=1+O(\hbar^2)$).
On the other hand, the map
$$F((M\otimes M)\otimes M)\xrightarrow{F(\id\otimes\epsilon\otimes\id)}F(M\otimes M)$$
becomes under these identifications
$$x\otimes y\mapsto xy,\quad U\g\otimes U\g\to U\g.$$
The product on $F(M\otimes M)\cong U\g$ is thus $x\otimes y\mapsto xy+O(\hbar^2)$.

Let us now compute the coproduct $\Delta_\hbar$ on $F(M\otimes M)\cong U\g$ to first order in $\hbar$. Let us recall that $M\otimes M$ is a coalgebra in $\D_\hbar^\Phi$, with the coproduct given by \eqref{eq:fus} (with $M'=M$). For $x\in\g$ we get
$$\Delta_{M\otimes M}(1\otimes x)=1\otimes x\otimes1\otimes 1+1\otimes1\otimes1\otimes x-\frac{\hbar}{2}\,1\otimes\delta(x)\otimes1+O(\hbar^2).$$
The coproduct $\Delta_\hbar$ on $F(M\otimes M)\cong U\g$ is thus
$$\Delta_\hbar(x)=x\otimes1+1\otimes x+\frac{\hbar}{2}\,\delta(x)+O(\hbar^2)$$
(where the sign change comes from $(\id\otimes S_0)(\delta(x))=-\delta(x)$),
hence
$$(\Delta^{\vphantom{op}}_\hbar-\Delta^{op}_\hbar)(x)=\hbar\,\delta(x)+O(\hbar^2),$$
as we wanted to show.
\end{proof}
Let us stress that the Hopf algebra structure on $F(M\otimes M)\cong U\g$ depends on the associator $\Phi$. Curiously, the antipode is independent of $\Phi$.

\begin{rem}[Quantization in terms of PROPs]
Etingof and Kazhdan proved in \cite{EK2} that their quantization is given by ``universal formulas'' in the following sense (which, in particular, implies functoriality of their quantization). Let $\on{LieBialg}$ be the PROP of Lie bialgebras and let
$\on{HcP}:=S(\on{LieBialg})$ ($\on{HcP}$ is the PROP of Hopf co-Poisson algebras). Let us denote by $\g_\textit{univ}$ the generating object of $\on{LieBialg}$ ($\g_\textit{univ}$ is the universal Lie bialgebra) and by $S\g_\textit{univ}$ the generating object of $\on{HcP}$. Then there is a Hopf algebra structure $(m_\hbar,\Delta_\hbar,S_\hbar,\eta_0,\epsilon_0)$ on $S\g_\textit{univ}$ (i.e.\ a PROP morphism from the PROP of Hopf algebras to $\on{HcP}$, formally depending on $\hbar$) which is a deformation of the universal enveloping algebra 
$$U\g_\textit{univ}=(S\g_\textit{univ},m_0,\Delta_0,S_0,\eta_0,\epsilon_0)$$
 in the direction of the cobracket of $\g_\textit{univ}$.
 
 The same is true for our quantization. Let $M:=S\g_\textit{univ}$ with its $\g_\textit{univ}$-dimodule structure as above. Let $\D$ be the full monoidal subcategory of the category of $\g_\textit{univ}$-dimodules in $\on{HcP}$, generated by $M$ (i.e.\ the objects of $\D$ are $1,M,M^{\otimes2},\dots$); the category $\D$ is i-braided and $M$ is an i-cocommutative coalgebra in $\D$. Let $\C=\on{HcP}$.
 
 Finally we need a $M$-adapted functor $F:\D\to\C$, which is simply 
 $$F(M^{\otimes k})=(S\g_\textit{univ})^{\otimes(k-1)}$$
  (``$\g_\textit{univ}$-coinvariants'').
Theorem \ref{thm:main} and Proposition \ref{prop:drinf} now produce a Hopf algebra structure on $S\g_\textit{univ}$.
\end{rem}

\begin{rem}[Quantization of infinitesimally braided Lie bialgebras]
If we apply Theorem \ref{thm:main} and Proposition \ref{prop:drinf} to Example \ref{ex:blba}, we make
$$F(M\otimes M)\cong U(\mf p^\perp)$$
to a Hopf algebra in the braided monoidal category $U\h\text{-mod}_\hbar^\Phi$, deforming the standard Hopf algebra structure on $U(\mf p^\perp)$. This is a special case of quantization of Lie bialgebras in (Abelian) i-braided categories, which is a minor generalization of what we did in above.

By a Lie bialgebra in  an i-braided category $\C$ we mean an object $\g$ in $\C$, together with a Lie bracket $\mu:\g\otimes\g\to\g$ and a Lie cobracket $\delta:\g\to\g\otimes\g$ such that
$$\delta\circ\mu=(\mu\otimes\id+(\id\otimes\mu)\circ\sigma_{12})\circ(\id\otimes\delta)
\circ(\id-\sigma_{12})+t_{\g,\g}\circ(\id-\sigma_{12})/2$$
(besides the $t_{\g,\g}$-term, it is the standard definition of a Lie bialgebra). As an example, $\mf p^\perp$ is a Lie bialgebra in the i-braided category $U\h$-mod, with the cobracket given by the Lie bracket on $\bar{\mf p}^\perp$.

 Let $\D$ be the category whose objects are objects $V$ of $\C$ together with a left action $\rho:\g\otimes V\to V$ and a right coaction $\tilde\rho:V\to V\otimes\g$, such that
$$(\id\otimes\rho)\circ\sigma_{1,2}\circ(\id\otimes\tilde\rho)=\tilde\rho\circ\rho+
(\id\otimes\rho)\circ(\delta\otimes\id)-
(\id\otimes\mu)\circ(\id\otimes\tilde\rho)+ t_{\g,V}.$$
Category $\D$ is i-braided, with
$$t^\D_{V,W}=(\id_V\otimes\rho_W)\circ\sigma_{1,2}\circ(\tilde\rho_V\otimes\id_W)+(\rho_V\otimes\id_W)\circ\sigma_{1,2}\circ(\id_V\otimes\tilde\rho_W)+t^\C_{V,W}.$$

Supposing that $\C$ is an Abelian category, so that we can make sense of $U\g$ and of $\g$-coinvariants, we define $M=U\g\in\D$ and $F:\D\to\C$ as above, and $F(M\otimes M)$ becomes a Hopf algebra in $\C_\hbar^\Phi$ deforming $U\g$.
\end{rem}

\begin{rem}[Quantization of twists]
The quantization of Lie bialgebras given by Theorem \ref{thm:qlb} is compatible with twists in the following sense. Suppose that $j\in\bigwedge^2\g$ is a twist of the Lie bialgebra $(\g,[,],\delta)$, i.e.\ that
$$\g^*_j:=\{j(\alpha,\cdot)+\alpha;\,\alpha\in\g^*\}\subset\dd$$
is a Lie subalgebra of the Drinfeld double $\dd$. Since $(\g,\g^*_j,\dd)$ is a Manin triple, we get a new Lie bialgebra structure on $\g$, with the original bracket and with the new cobracket $\delta_j(u)=\delta(u)+[1\otimes u+u\otimes1,j]$. Let $H$ be the quantization of $(\g,[,],\delta)$ and $H_j$ the quantization of $(\g,[,],\delta_j)$ (as given by Theorem \ref{thm:qlb}; we have $H=H_j=U\g$ as vector spaces). Then there is a twist $J\in H\otimes H[[\hbar]]$, i.e.\ an  element satisfying
$$J^{12,3}J^{1,2}=J^{1,23}J^{2,3}\quad\text{and}\quad (\epsilon\otimes\id)(J)=(\id\otimes\epsilon)(J)=1$$
(where $J^{12,3}=(\Delta\otimes\id)(J)$, $J^{1,2}=J\otimes1$, etc.),
such that $J=1+O(\hbar)$ and $(J-J^{op})=\hbar j+O(\hbar^2)$.
Moreover, there is an isomorphism of Hopf algebras
$$I:H^{(J)}\cong H_j,\quad I=\id+O(\hbar^2),$$
where $H^{(J)}$ has the same product, unit and counit as $H$, and the coproduct 
$$\Delta_{H^{(J)}}(a)=J^{-1}\Delta_H(a)J.$$

This statement was proven for Etingof-Kazhdan quantization by Enriquez and Halbout \cite{EH}, but their proof required a considerable effort. In our case it is a consequence of the following observation. If, in the context of Theorem \ref{thm:main}, $N\in\D$ is another cocommutative coalgebra, and $F:\D\to\C$ is both $M$- and $N$-adapted, then the coalgebra $F(M\otimes N)$ is a $F(M\otimes M)\,$-$\,F(N\otimes N)$ bimodule in the category of coalgebras in $\C$, with the action given by
 \begin{gather*}
F(M\otimes M)\otimes F(M\otimes N)\xrightarrow{{\tau_{M,N}^{(M)}}^{-1}}
 F((M\otimes M)\otimes N)\xrightarrow{F(\id_M\otimes\epsilon_M\otimes\id_N)}F(M\otimes N)\\
F(M\otimes N)\otimes F(N\otimes N)\xrightarrow{{\tau_{M,N}^{(N)}}^{-1}} 
 F((M\otimes N)\otimes N)\xrightarrow{F(\id_M\otimes\epsilon_N\otimes\id_N)}F(M\otimes N).
\end{gather*}

To apply it, let $\D,\C,M,F$ as in the proof of Theorem \ref{thm:qlb} (in particular, $\D$ is the category of $(\g,[,],\delta)$-dimodules), and let $\D_j,M_j,F_j$ be the corresponding objects for the Lie bialgebra $(\g,[,],\delta_j)$ in place of $(\g,[,],\delta)$. Notice that the categories $\D$ and $\D_j$ are naturally isomorphic (essentially because $\g$ and $\g_j$ have the same double $\dd$): if 
$$(V,\rho:\g\otimes V\to V,\tilde\rho:V\to V\otimes\g)$$
is a $(\g,[,],\delta)$-dimodule, then $(V,\rho,\tilde\rho_j)$ is  a $(\g,[,],\delta_j)$-dimodule, with
$$\tilde\rho_j(v)=\tilde\rho(v)+\sum_ij^1_i\otimes\rho(j^2_i\otimes v)$$
where $j=\sum_i j^1_i\otimes j^2_i$. If we identify $\D_j$ with $\D$ using this isomorphism, we get $\D_j=\D$, $F_j=F$, while $M_j$ becomes the dimodule $N\in\D$, $N=U\g$,
 with the coaction $\tilde\rho:N\to N\otimes\g$ uniquely determined by
$$\tilde\rho(1)=j.$$
We have $H=F(M\otimes M)$ and $H_j=F(N\otimes N)$.  Let us denote by $B$ the $H\,$-$\,H_j$ bimodule $F(M\otimes N)$ (in the category of coalgebras in $\mathcal Vect_\hbar$).

Notice that $H=H_j=B=U\g$ as vector spaces, and that the action $H\otimes B\otimes H_j\to B$ is of the form $a\otimes b\otimes c\mapsto a\cdot b\cdot c= abc+O(\hbar^2)$. Now we can define the twist $J\in H\otimes H[[\hbar]]$  by
\begin{equation}\label{eq:J-B}
J\cdot(1\otimes 1)=\Delta_B 1\quad (1\otimes 1\in B\otimes B, 1\in B)
\end{equation}
and the isomorphism $I:H^{(J)}\cong H_j$ by
$$a\cdot 1=1\cdot I(a)\quad (1\in B, a\in H, I(a)\in H_j).$$ The fact that $J$ is a twist follows from coassociativity of $\Delta_B$: the equality 
$$(\Delta_B\otimes\id)\Delta_B 1=(\id\otimes\Delta_B)\Delta_B 1$$ and \eqref{eq:J-B} give $$J^{12,3}J^{1,2}\cdot1\otimes1\otimes1=J^{1,23}J^{2,3}\cdot1\otimes1\otimes1.$$
\end{rem}

\section{Quantization of Poisson-Lie groups}

In this section we dualize our quantization of Lie bialgebras to get a deformation quantization of Poisson-Lie groups. This translation is straightforward (replacing coalgebras with algebras and comonoidal functors with monoidal functors), but we describe the resulting star-product with some detail, to reveal the geometric intuition behind the algebraic constructions. Our quantization of Poisson-Lie groups can be seen a special case of the deformation quantization of moduli spaces of flat connections studied in \cite{LBS},  though this particular quantization of Poisson-Lie groups was missed in \emph{op.~cit.}

Let $G$ be a Poisson-Lie group.
Let $\g$ be its Lie algebra and $(\g,\g^*\subset\dd$) the corresponding Manin triple. There is an action $\rho$ of $\dd$ on $G$ (the ``dressing action''), given by
\begin{alignat}{2}
\rho(v)&=v_L &\quad&\text{ for } v\in\g\\
\rho(\alpha)&=\pi(\cdot,\alpha_L) &\quad&\text{ for } \alpha\in\g^*\!,
\end{alignat}
where $\pi$ is the Poisson bivector field on $G$ and $x_L$ means $x$ left-transported over $G$. 
For any $f_1,f_2\in C^\infty(G)$ we then have
$$\{f_1,f_2\}=\sum_i(\rho(e^i)f_1)\,(\rho(e_i)f_2),$$
where $e_i$ is a basis of $\g$ and $e^i$ the dual basis of $\g^*$. The stabilizer of $g\in G$ is $\on{Ad}_g\g^*\subset\dd$.

The action $\rho$ makes  $C^{\infty}(G)$ to an algebra in $U\dd\text{-mod}$. This algebra is i-commuta\-tive:
 $$m\circ \rho^{\otimes2}(t)=0,$$
  where $m$ is the product, as the stabilizers of $\rho$ are coisotropic. As a result, $C^{\infty}(G)$, with its original product, is a commutative associative algebra in $U\dd\text{-mod}_\hbar^\Phi$.

We can now make $C^\infty(G\times G)$ to an associative (but not commutative) algebra in $U\dd\text{-mod}_\hbar^\Phi$. $C^\infty(G\times G)=C^\infty(G)\mathop{\hat\otimes} C^\infty(G)$, being a (completed) tensor product of algebras, is again an algebra in $U\dd\text{-mod}_\hbar^\Phi$, with the product
$$m_\hbar=
\begin{tikzpicture}[baseline=1cm]
\coordinate (diff) at (0.2,0);
\coordinate (dy) at (0,0.5);
\node(A1) at (0,0) {};
\node(B1) at ($(A1)+(diff)$) {};
\node(A2) at (2,0) {};
\node(B2) at ($(A2)+(diff)$) {};
\node(A3) at ($(A1)+(0,2.5)+0.5*(diff)$) {};
\node(B3) at ($(A2)+(0,2.5)+0.5*(diff)$) {};
\coordinate(A) at ($(A3)-(dy)$);
\coordinate(B) at ($(B3)-(dy)$);
\draw[line width=1ex,white]  (A2)..controls +(0,1) and +(0,-1)..(A);
\draw (A2)..controls +(0,1) and +(0,-1)..(A);
\draw[line width=1ex,white] (B1)..controls +(0,1) and +(0,-1)..(B);
\draw (B1)..controls +(0,1) and +(0,-1)..(B);

\draw (A1)..controls +(0,1) and +(0,-1)..(A);
\draw (B2)..controls +(0,1) and +(0,-1)..(B);
\draw (B)--(B3);
\draw(A)--(A3);
\end{tikzpicture}
=m_0\circ(\mathcal T\cdot),
$$
where $m_0$ is the ordinary product on $C^\infty(G\times G)$ and 
 $$\mathcal T \in U\dd^{\otimes 4}[\![\hbar]\!]$$
is given by the parenthesized braid \eqref{eq:parb}.

Observe finally that $C^\infty(G\times G)^\g\subset C^\infty(G\times G)$ with the product $m_\hbar$ is  an associative algebra in $\mathcal Vect_\hbar$ (as the functor $X\mapsto X^\g$, $U\dd\text{-mod}_\hbar^\Phi\to\Vect_\hbar$ is monoidal).

\begin{prop}
If we identify $C^\infty(G)$ with $C^\infty(G\times G)^\g$ via $\tilde f(g_1,g_2)=f(g_1^{\vphantom{1}}g_2^{-1})$ ($f\in C^\infty(G)$, $\tilde f\in C^\infty(G\times G)^\g$) then $m_\hbar$ becomes a star-product on $G$ quantizing the Poisson structure $\pi$.
\end{prop}
\begin{proof}
The fact that $m_\hbar$ is a star-product, i.e.\ that it is associative and that its coefficients in the $\hbar$-power series are bidifferential operators, is clear from the construction. We need to verify that $$m_\hbar(f_1,f_2)-m_\hbar(f_2,f_1)=\hbar\{f_1,f_2\}+O(\hbar^2).$$
Since $\mathcal T=1-\hbar\,t^{2,3}/2+O(\hbar^2)$, we get
$$m_\hbar(f_1,f_2)=f_1f_2-\frac{\hbar}{2}\sum_i(\rho(e_i)f_1)\,(\rho(e^i)f_2)+O(\hbar^2)=f_1f_2+\frac{\hbar}{2}\{f_1,f_2\}+O(\hbar^2).$$
\end{proof}

Theorem \ref{thm:main} (in its dualized version, with the algebra $C^\infty(G)$ in place of a coalgebra $M$, and the monoidal functor $X\mapsto X^\g$ in place of a comonoidal functor $F$) gives us a coproduct
$$\Delta_\hbar:C^\infty(G)\to C^\infty(G\times G)$$
which, together with $m_\hbar$,  makes $C^\infty(G)$ to a (topological) bialgebra in $\mathcal Vect_\hbar$. The coproduct is of the form.
$$\Delta_\hbar=D_\hbar\circ\Delta_0$$
where $\Delta_0:C^\infty(G)\to C^\infty(G\times G)$ is the ordinary coproduct, 
$$(\Delta_0 f)(g_1,g_2)=f(g_1g_2),$$ and
$$D_\hbar=1+\hbar^2 D_2+\hbar^3 D_3+\dots$$
is a differential operator  on $G\times G$.

Let us conclude by recalling  the proof of Theorem \ref{thm:main} in our current context, to see its geometrical meaning. For any $k\in\mathbb N$ we have
$$C^\infty(G^{k+1})=
\underbrace{C^\infty(G)\mathop{\hat\otimes}\dots\mathop{\hat\otimes}C^\infty(G)}_{k+1}.$$
By viewing it as a tensor product of algebras in  $U\dd\text{-mod}_\hbar^\Phi$ we make $C^\infty(G^{k+1})$ to an associative algebra in this category (we need to choose a parenthesization of the tensor product, but all the choices are canonically isomorphic). We can view this sequence of algebras as a cosimplicial algebra (namely the bar construction of the commutative algebra $C^\infty(G)$); we thus get a deformation quantization (in $U\dd\text{-mod}_\hbar^\Phi$)  of the simplicial manifold $E_\bullet G=G^{\bullet+1}$. By taking the $\g$-invariants we get a deformation quantization (in $\mathcal Vect_\hbar$) of the nerve $B_\bullet G=(E_\bullet G)/G$, which is finally the cobar construction of the resulting Hopf algebra.

\end{document}